\documentclass[12pt]{amsart}
\usepackage{amscd,amsmath,amsthm,amssymb}
\usepackage[document]{ragged2e}
\usepackage{amsmath, amsthm, amscd, amsfonts, amssymb, enumerate, verbatim, newlfont, calc, graphicx, color}
\usepackage[bookmarksnumbered, colorlinks, plainpages]{hyperref}
\usepackage{tikz}
\usepackage{doi}
\usepackage{graphicx}
\usepackage[margin=2.5cm]{geometry}

\usepackage{pstcol,pst-plot,pst-3d}
\usepackage{lmodern,pst-node}
\usepackage[normalem]{ulem}
\usepackage{hyperref}
\psset{unit=0.7cm,linewidth=0.8pt,arrowsize=2.5pt 4}

\newpsstyle{fatline}{linewidth=1.5pt}
\newpsstyle{fyp}{fillstyle=solid,fillcolor=verylight}
\definecolor{verylight}{gray}{0.97}
\definecolor{light}{gray}{0.93}
\definecolor{medium}{gray}{0.82}
 \def\NZQ{\Bbb}               
 \def\NN{{\NZQ N}}

 \def\frk{\frak}               

 \def\mm{{\frk m}}

 \def\G{{\mathcal G}}

 \def\H{{\mathcal H}}

 \def\opn#1#2{\def#1{\operatorname{#2}}} 
 \opn\chara{char} \opn\length{\ell} \opn\pd{pd} \opn\rk{rk}
 \opn\projdim{proj\,dim} \opn\injdim{inj\,dim} \opn\rank{rank}
 \opn\depth{depth} \opn\grade{grade} \opn\height{height}
 \opn\embdim{emb\,dim} \opn\codim{codim}
 
 \opn\Tr{Tr} \opn\bigrank{big\,rank}
 \opn\superheight{superheight}\opn\lcm{lcm}
 \opn\trdeg{tr\,deg}
 \opn\reg{reg} \opn\lreg{lreg} \opn\ini{in} \opn\lpd{lpd}
 \opn\size{size} \opn\sdepth{sdepth}
 \opn\link{link}\opn\fdepth{fdepth}\opn\lex{lex}
 \opn\div{div} \opn\Div{Div} \opn\cl{cl} \opn\Cl{Cl}
 \opn\Spec{Spec} \opn\Supp{Supp} \opn\supp{supp} \opn\Sing{Sing}
 \opn\Ass{Ass} \opn\Min{Min}\opn\Mon{Mon}
 \opn\Ann{Ann} \opn\Rad{Rad} \opn\Soc{Soc}
 \opn\Im{Im} \opn\Ker{Ker} \opn\Coker{Coker} \opn\Am{Am}
 \opn\Hom{Hom} \opn\Tor{Tor} \opn\Ext{Ext} \opn\End{End}
 \opn\Aut{Aut} \opn\id{id}
 
 \opn\nat{nat}
 \opn\pff{pf}
 \opn\Pf{Pf} \opn\GL{GL} \opn\SL{SL} \opn\mod{mod} \opn\ord{ord}
 \opn\Gin{Gin} \opn\Hilb{Hilb}\opn\sort{sort}
 \opn\aff{aff} \opn
 \con{conv} \opn\relint{relint} \opn\st{st}
 \opn\lk{lk} \opn\cn{cn} \opn\core{core} \opn\vol{vol}
 \opn\link{link} \opn\star{star}\opn\lex{lex}\opn\set{set}

 \opn\Ker{Ker} \opn\HS{HS} \opn\G{G}

 \opn\gr{gr}
 
 \def\pot#1#2{#1[\kern-0.28ex[#2]\kern-0.28ex]}

 \opn\dirlim{\underrightarrow{\lim}}
 \opn\inivlim{\underleftarrow{\lim}}
 \def\Implies{\ifmmode\Longrightarrow \else
         \unskip${}\Longrightarrow{}$\ignorespaces\fi}
 \def\implies{\ifmmode\Rightarrow \else
         \unskip${}\Rightarrow{}$\ignorespaces\fi}
 \def\iff{\ifmmode\Longleftrightarrow \else
         \unskip${}\Longleftrightarrow{}$\ignorespaces\fi}

 \let\:=\colon
 \newtheorem{Theorem}{Theorem}[section]
 
 \newtheorem{Corollary}[Theorem]{Corollary}
 \newtheorem{Proposition}[Theorem]{Proposition}
 \newtheorem{Remark}[Theorem]{Remark}
 
 \newtheorem{Example}[Theorem]{Example}
 
 \newtheorem{Definition}[Theorem]{Definition}

 \let\epsilon\varepsilon
 \let\kappa=\varkappa
 \def\qed{\ifhmode\textqed\fi
       \ifmmode\ifinner\quad\qedsymbol\else\dispqed\fi\fi}
 \def\textqed{\unskip\nobreak\penalty50
        \hskip2em\hbox{}\nobreak\hfil\qedsymbol
        \parfillskip=0pt \finalhyphendemerits=0}
 \def\dispqed{\rlap{\qquad\qedsymbol}}
 
 \opn\dis{dis}
 \def\pnt{{\raise0.5mm\hbox{\large\bf.}}}
 
 \opn\Lex{Lex}

\begin{document}
\justifying
 \title {Componentwise linear ideals and exchange properties}

\author{ Ayesha Asloob Qureshi}
\address{Sabanci University, Faculty of Engineering and Natural Sciences, Orta Mahalle, Tuzla 34956, Istanbul,
Turkey} \email{ayesha.asloob@sabanciuniv.edu and aqureshi@sabanciuniv.edu}

\author{Somayeh Bandari}
\address{Department of Mathematics, Buein Zahra Technical University, Buein Zahra, Qazvin, Iran}
\email{somayeh.bandari@yahoo.com and s.bandari@bzte.ac.ir}

\thanks{Ayesha Asloob Qureshi is supported by Scientific and Technological Research Council
of Turkey T\"UB\.{I}TAK under the Grant No: 122F128, and is thankful to T\"UB\.{I}TAK for their supports}

 \begin{abstract}
We prove the componentwise linearity of ideals that satisfy a certain exchange property similar to polymatroidal ideals. We also discuss the componentwise linearity and exchange properties of ideals of $k$-covers of totally balanced weighted hypergraphs. 
 \end{abstract}

 \maketitle

  \section*{Introduction}
A graded ideal $I \subset S=K[x_1, \ldots, x_n]$ is
called {\em componentwise linear} if $I_{\langle d \rangle}$ admits
a linear resolution for all $d$, where $I_{\langle d \rangle}$ denotes the ideal generated by all homogeneous elements of degree
$d$ of $I$. Componentwise linear ideals were introduced by Herzog and Hibi in
\cite{HH1}, and they appear
naturally in combinatorial and algebraic geometrical contexts; see
\cite{FT} for insight into such examples. In \cite{HH1}, it is shown
that the Stanley-Reisner ideal $I_\Delta$ of a simplicial complex
$\Delta$ is componentwise linear if and only if the Alexander dual
of $\Delta$ is sequentially Cohen-Macaulay. This result
combinatorially characterizes squarefree componentwise linear
ideals. A useful tool to prove componentwise linearity of an ideal
is the notion of linear quotients. Ideals with linear
quotients were introduced by Herzog and Takayama in \cite{HT}. In
\cite[Theorem 2.7]{JZ}, it is proved that if a monomial ideal $I$
has linear quotients, then it is componentwise linear. This result
is one of  the most effective tools while studying the componentwise
linearity of monomial ideals, and it has been applied in several
articles in this context; for example, \cite{BH}, \cite{BQ},
\cite{FT}, \cite{MM}.

In this paper, we discuss componentwise linearity of monomial ideals
whose minimal monomial generators satisfy certain exchange
relations. The ideals with the non-pure dual exchange property (see
Definition~\ref{def:dual}) were introduced in the previous work \cite{BQ} by the
authors, to discuss the linear quotients property of
componentwise polymatroidal ideal. A monomial ideal $I$ is called
{\em polymatroidal} if the set of exponent vectors of the minimal
monomial generators of $I$ is the set of bases of a discrete
polymatroid; see \cite{HHdis} for more information. Given that the set of bases of a discrete polymatroid is distinguished by the ``exchange property",
 it follows that a polymatroidal ideal can also be characterized accordingly. Let $I \subset S=K[x_1, \ldots, x_n]$ be a monomial
ideal generated in a single degree. A monomial ideal $I$ is called {\em polymatroidal} if its monomial generators enjoy the following exchange property: for any two
minimal monomial generators $u, v \in I$, if $\deg_{x_i} (u)>
\deg_{x_i}(v)$, then  there exists some $j$ with $\deg_{x_j}
(u)<\deg_{x_j}(v)$ such that $x_j(u/x_i) \in I$. In other words,
after exchanging $x_i$ with $x_j$ in $u$, we again obtain an element
of $I$. We will refer to such properties of minimal monomial
generators of a monomial ideal as {\em exchange properties}.
Polymatroidal ideals have been studied by many authors and hold a
special place among monomial ideals due to their well-tamed
algebraic and homological behavior, which is mainly due to the exchange property satisfied by their generators. For example, it is known that the product of polymatroidal ideals is again polymatroidal, polymatroidal ideals have linear quotients and hence all powers of a polymatroidal ideal have a linear resolution, see \cite{CH}.  In \cite{BH}, authors formulated the definition of {\em componentwise polymatroidal} ideals to introduce a class of monomial ideals that are
not necessarily generated in a single degree and share some of the
nice properties that polymatroidal ideals enjoy. A monomial ideal
$I$ is called componentwise polymatroidal if $I_{\langle d \rangle}$
is polymatroidal for all $d>0$. It directly follows from the definition of
componentwise polymatroidal ideal that they are componentwise linear. 
It is natural to ask that if minimal monomial generators of a
componentwise polymatroidal ideal also satisfy some exchange
property similar to that of polymatroidal ideals.  
In \cite[Proposition 1.2]{BQ}, the authors answered this question
affirmatively, and in \cite[Proposition 1.5]{BQ}, they showed that
these ideals satisfy the non-pure dual exchange property. Recently, Ficarra
in \cite{F} proved a conjecture of Bandari and Herzog stating that
``all componentwise polymatroidal ideal have linear quotients", and
in his proof, it was observed that the non-pure dual exchange
property of componentwise polymatroidal ideals plays a vital role.
Our aim in this work is to find new classes of componentwise linear monomial ideals 
whose minimal monomial generators satisfy certain exchange property, particularly 
the non-pure dual exchange property.

The breakdown of this paper is as follows: In Section~\ref{sec1}, we
prove that ideals with the non-pure dual exchange property are
componentwise linear by showing that they admit linear quotients. Our
proof follows the same ordering of generators as in work of Ficarra
in \cite{F}. Unlike polymatroidal ideals, we observe in Example~\ref{ex:produt-dual}
that powers of ideals with the non-pure dual exchange property do not necessarily belong to the same class. However, in Theorem~\ref{thm:linquotient}, we show that if $I$ has the non-pure dual exchange property, then so does $\mm I $, where
$\mm$ is the unique graded maximal ideal of $S$.  In Section~\ref{sec2}, we
discuss different classes of ideals of $k$-covers $I_k(\mathcal{H}, \omega)$ of weighted hypergraphs $(\mathcal{H}, \omega)$, also known as intersection of Veronese ideals, see \cite{FT}. Here $\H$ is a hypergraph equipped with a weight function $\omega$; see Section~\ref{sec2} for a formal definition. Mohammadi and Moradi in \cite{MM}, and Francisco and Tyul in \cite{FT}, studied the componentwise linearity of $I_k(\mathcal{H}, \omega)$. In \cite{FT}, for different classes of $\H$, the authors showed that $I_k(\mathcal{H}, \omega)$ either has linear quotients or it is componentwise polymatroidal. On the other hand, in \cite{MM}, authors showed that for different classes of $\H$, the ideal $I_k(\mathcal{H}, \omega)$  is either weakly polymatyroidal or componentwise weakly poymatroidal. The property of an ideal being componentwise polymatroidal depends on its ambient ring, as noted in \cite[Remark 3.3]{FT}, whereas the property of being weakly polymatroidal depends on the ordering of the variables. The definition of the non-pure dual exchange property is independent of these restrictions. In Section~\ref{sec2}, our aim is to extend work done in \cite{FT} and \cite{MM}. We investigate under what conditions on the edges of $\mathcal{H}$ the ideal of $k$-covers of $(\H,\omega)$ satisfies the non-pure dual exchange property. We particularly focus on totally balanced hypergraphs, also known as simplicial forest defined by Faridi in \cite{Far2}. We prove the non-pure dual exchange property for the ideal of $k$-covers of weighted totally balanced hypergraph $(\H, \omega)$ on vertex set $[n]=\{1, \ldots, n\}$ and edges 
\begin{enumerate}
	\item (Theorem~\ref{thm:intersection}) $J_1,\ldots, J_s, K$  such that $J_i\cap
	J_j=\cap_{t=1}^s J_t$ for all $1\leq i<j\leq s$ and $K\cap
	J_t=\emptyset$ for all $t=1, \ldots, s$,
	
	\item  (Theorem~\ref{thm:4ideals}) $J_1,J_2,J_3, J_4$ such that $J_1 \cap J_3=J_1\cap J_4 = J_2\cap
	J_4=\emptyset$, and 	$J_2 \subseteq J_1 \cup J_3$ and $J_3 \subseteq J_2 \cup J_4$. In this case we let $\omega(J_i)=1$, for all $i$.
\end{enumerate}
We also prove that the ideal of $k$-covers of weighted totally balanced hypergraph with edges  $J_1,J_2, J_3\subset [n]$ such that $J_1 \cap J_3 = \emptyset$ is weakly polymatroidal if $\omega(J_1) \geq \omega(J_2) \geq \omega(J_3)$; see Theorem~\ref{thm:3weakly}. We also proivde examples to justify the conditions on the edges of these hypergraphs. 

\section{Ideals with the non-pure dual exchange property}\label{sec1}
The ideals with the non-pure dual exchange property were defined in \cite{BQ} to facilitate the study of componentwise polymatroidal ideals. In this section, our main goal is to show that the ideals with the non-pure dual exchange property are componentwise linear ideals. To this end, we first recall some required definitions and notions. Let $S=K[x_1, \ldots, x_n]$ be a polynomial ring over a field $K$. For any monomial ideal  $I \subset S$, we denote the unique minimal generating set of $I$ by $G(I)$. 

\begin{Definition} \label{def:dual} A monomial ideal $I$ is said to satisfy the {\em non-pure
dual exchange property} if, for all $u, v\in G(I)$ with $\deg(u)\leq
\deg(v)$ and for all $i$ such that $\deg_{x_i}(v) < \deg_{x_i}(u)$,
there exists $j$ such that $\deg_{x_j}(v) > \deg_{x_j}(u)$ and
$x_i(v/x_j)\in I$.
\end{Definition}

By virtue of \cite[Theorem 2.7]{JZ}, to show that an ideal with the non-pure dual exchange property is componentwise linear, it is enough to show that it admits linear quotients. A monomial ideal $I\subset S$ is said to have {\em linear quotients} if there exists an ordering $u_1,\ldots,u_m$ of its
minimal generators such that, for each $i=2,\ldots,m$, the colon
ideal $(u_1,\ldots,u_{i-1}):u_i$ is generated by a subset of the
variables. Any such ordering of $G(I)$ is called an {\em admissible order}. Below we show that ideals with the non-pure dual exchange property admit linear quotients. The admissible order in the proof below is similar to the one introduced by Ficarra in \cite{F} to prove the componentwise linearity of componentwise polymatroidal ideals. 

\begin{Theorem}\label{thm:linquotient}
Let $I$ be an ideal with the non-pure dual exchange property. Then $I$
has linear quotients.
\end{Theorem}

\begin{proof}
 We prove the theorem by induction on $|G(I)|$. If $|G(I)|=1$,  then $I$ is a principal ideal and it has linear quotients.
 Now, let $|G(I)|>1$. We may assume  that all monomials $u\in G(I)$ do not have  common factor $w\neq 1$. Otherwise, we may consider
the ideal $I'$ with $G(I')=\{u/w \;|\; u\in G(I)\}$. Then $I'$ also has the
non-pure exchange property, and $I$ has linear quotients if and
only if $I'$ has linear quotients. Let $d=\min \{\deg(u) : u \in
G(I)\}$. After a suitable relabeling, we may assume that $x_1\in
\supp(I_{\langle d \rangle })$. Therefore, we have  $I=x_1I_1+I_2$,
where $x_1\nmid u$ for $u\in G(I_2)$. Note that $G(I)=G(x_1 I_1)
\cup G(I_2)$. We claim that $I_2 \subseteq I_1$ and $x_1I_1$ and
$I_2$ are ideals with the non-pure dual exchange property.

Proof of $I_2 \subseteq I_1$: Let $v\in G(I_2)$ and let $u\in
x_1I_1$ with $\deg(u)=d$. Therefore, $\deg(u)\leq\deg(v)$. Moreover,
$\deg_{x_1}(v)=0 <\deg_{x_1}(u)$. By the definition of the non-pure dual
exchange property, there exists an index $j$ such that
$\deg_{x_j}(v)>\deg_{x_j}(u)$ and $x_1(v/x_j)\in I$. Thus, there is
$w\in G(I)$ such that $w|x_1(v/x_j)$. If $w\in G(I_2)$, then
$x_1\nmid w$, so $w$ divides $v/x_j$, which contradicts $v\in G(I)$. Hence $w \in G(x_1I_1)$. So $w=x_1w'$ for $w'\in
I_1$. Therefore, $w'$ divides $v/x_j$. Hence, $w'|v$, as desired.

To see $x_1 I_1 $ has the non-pure dual exchange property, take $u,v \in
G(x_1I_1) \subset G(I)$ with $\deg(u)\leq \deg(v)$ and let $i$ be such
that $\deg_{x_i}(v) < \deg_{x_i}(u)$. Since $I$ has the non-pure dual
exchange property and $u,v \in G(I)$, there  exists $j$ such that
$\deg_{x_j}(v) > \deg_{x_j}(u)$ and $w=x_i(v/x_j)\in I$.  If $j\neq
1$, then $x_1$ divides $w$ because $x_1$ divide $v$. If $j=1$, then
$\deg_{x_1}(v) > \deg_{x_1}(u)\geq 1$, which means that $x_1$ divides $w$.
Since $w \in I$, there exists some $w' \in G(I)= G(x_1I_1) \cup
G(I_2)$ such that $w'$ divides $w$. If $w' \in G(x_1I_1)$, then $w
\in x_1I_1$, as required. If  $w'\in G(I_2)$, then $x_1$ does not
divide $w'$. Using the fact that $w'$ and $x_1$ divide $w$ and that $I_2 \subseteq
I_1$, we obtain $w \in x_1 I_2 \subseteq x_1 I_1$, as required.

Now, we want to show that $I_2$ has the non-pure dual exchange property.
Let $u, v\in G(I_2)\subset G(I)$ with $\deg(u)\leq \deg(v)$ and let $i$
be such that $\deg_{x_i}(v) < \deg_{x_i}(u)$. Then there exists $j$
such that $\deg_{x_j}(v) > \deg_{x_j}(u)$ and $x_i(v/x_j)\in I$.
Since $x_1$ does not divide $u$ and $v$, it follows that $x_1$ does not
divide $x_i(v/x_j)$. This shows that $x_i(v/x_j)\in I_2$.

Now, since $|G(x_1I_1)|$ and $|G(I_2)|$ are strictly less than
$|G(I)|$, and $x_1I_1$ and $I_2$ are monomial ideals with the non-pure
dual exchange property, it follows by our induction hypothesis that
$x_1I_1$ and $I_2$ have linear quotients. Let $x_1I_1$ has linear
quotients with the admissible order $u_1,\ldots,u_r$ and let $I_2$ has linear
quotients with the admissible order $v_1,\ldots,v_s$. We claim that $I$
has linear quotients with the admissible order
$u_1,\ldots,u_r,v_1,\ldots,v_s$. For each $l=2,\ldots,r$, the ideal
$(u_1,\ldots,x_{l-1}):u_l$ is generated by variables. Now, let
$l\in\{2,\ldots,s\}$. Then
\begin{equation*}
\begin{split}
(u_1,\ldots,u_r,v_1,\ldots,v_{l-1}):v_l&=(u_1,\ldots,u_r):v_l+(v_1,\ldots,v_{l-1}):v_l\\
&=(x_1I_1:v_l)+(v_1,\ldots,v_{l-1}):v_l\\
&=(x_1)+(v_1,\ldots,v_{l-1}):v_l,
\end{split}
\end{equation*}
is generated by variables. Note that since $v_l\in I_2\subseteq I_1$
and $x_1\nmid v_l$, we have that $(x_1I_1:v_l)=(x_1)$.
\end{proof}

Polymatroidal ideals are distinguished by the fact that the product of polymatroidal ideals remains polymatroidal. However, in the following example, we observe that the powers of ideals with the non-pure dual exchange property need not satisfy the non-pure dual exchange property.

\begin{Example}\label{ex:produt-dual}
	Let $I=(x_1^2,x_1x_2^2,x_1x_2x_3,x_2^2x_3,x_1x_3^3,x_2x_3^3)$.  It can be easily verified that $I$ has the non-pure dual exchange property. However, $I^2$ does not	have the non-pure dual exchange property. To see this, take
	$u=x_1^3x_3^3\in G(I^2)$ and $v=x_2^4x_3^2\in G(I^2)$.  Then
	$\deg_{x_3}(v)<\deg_{x_3}(u)$ and $x_3(v/x_2) =x_2^3x_3^3\not \in I^2$.
\end{Example}

On the other hand, ideals with the non-pure dual exchange property preserved their characteristic when multiplied with the unique graded maximal ideal of $S$, as shown in the following result. 

\begin{Proposition} \label{pro}
Let $I\subset K[x_1, \ldots, x_n]$ be an ideal with the non-pure dual exchange property, and $\mm$ be the unique graded maximal ideal of $S$. Then $\mm
I$ also has the non-pure dual exchange property.
\end{Proposition}
\begin{proof}
Let $x_ku,x_tv\in G(\mm I)$ such that $u,v \in G(I)$, and $\deg
(x_ku) \leq \deg (x_tv)$. We need to show that for each $i$ with
$\deg_{x_i}(x_tv)< \deg_{x_i} (x_ku)$, there exists $j$ such that
$\deg_{x_j}(x_tv)> \deg_{x_j} (x_ku)$ and $x_i(x_tv/x_j) \in \mm I$.
The inequality $\deg (x_ku) \leq \deg (x_tv)$ leads to $\deg (u)
\leq \deg (v)$.

First, suppose that   $\deg_{x_t}(x_tv)< \deg_{x_t} (x_ku)$. Since
$\deg_{x_t}(v)=\deg_{x_t}(x_tv)-1$, we obtain $\deg_{x_t}(v)<
\deg_{x_t} (u)$. Hence, there exists $j$ such that $\deg_{x_j}(v)>
\deg_{x_j} (u)$ and $v_1=x_t(v/x_j) \in  I$. If $\deg_{x_j}(x_tv)>
\deg_{x_j} (x_ku)$ then $x_tv_1=x_t(x_tv/x_j) \in \mm I$, as
required. Otherwise, if $\deg_{x_j}(x_tv) = \deg_{x_j} (x_ku)$, then $j=k$ and $j\neq t$. Also, $\deg_{x_j}
(v_1)=\deg_{x_j}(v)-1=\deg_{x_j}u$. Observe that $v_1 \in  G(I)$, otherwise, there exists $w \in G(I)$ with $\deg (w) < \deg(v_1)$
such that $w$ divides $v_1$, then $x_j w$ divides $x_t v$ in
$\mm I$, contradicting the assumption that $x_tv
\in G(\mm I )$. Furthermore, we also have $\deg_{x_t}(u)- \deg_{x_t}
(v) \geq 2$ because $\deg_{x_t}(x_tv)< \deg_{x_t} (x_ju)$. Hence
$\deg_{x_t} (v_1) < \deg_{x_t} (u)$. Then there exists some
$x_{\ell}$ such that $\deg_{x_{\ell}} (v_1) > \deg_{x_{\ell}}( u)$
and $x_t(v_1/x_{\ell}) \in I$. Since $l \neq j$ and $l \neq t$, we
conclude that $\deg_{x_{\ell}} (x_tv)=\deg_{x_{\ell}}
(v)=\deg_{x_{\ell}} (v_1)> \deg_{x_{\ell}}( x_ju)$. Furthermore,
$x_t(v_1/x_{\ell})= x_t(x_t v/x_jx_{\ell})\in I$ gives $x_j
(x_t(v_1/x_{\ell}))=x_t(x_t v/x_{\ell})\in \mm I$, as required.

Now let $\deg_{x_i}(x_tv)< \deg_{x_i} (x_ku)$, for some $i \neq t$.
Then $\deg_{x_i}(v)=\deg_{x_i}(x_tv)$.
We have following two cases to consider.\\

{\bf Case 1:} Let $i \neq k$. Then $\deg_{x_i} (u) = \deg_{x_i}
(x_ku)$, and $\deg_{x_i}(v)< \deg_{x_i} (u)$. There exists $x_j$
such that  $\deg_{x_j}(v)>\deg_{x_j} (u)$ and $x_i(v/x_j)\in I$. If
$\deg_{x_j}(x_tv)> \deg_{x_j} (x_ku)$, then $x_i(x_tv/x_j) \in \mm
I$, as required. Otherwise, $\deg_{x_j}(x_tv)= \deg_{x_j} (x_ku)$,
which is the case if and only if $\deg_{x_j}(v)-1=\deg_{x_j} (u)$,
$j=k$ and $j \neq t$. In this case, we compare $\deg_{x_t}(x_tv)$
and $ \deg_{x_t} (x_ju)$. If $\deg_{x_t}(x_tv)> \deg_{x_t} (x_ju)$,
then $x_i(x_tv/x_t) \in \mm I$, and we are done.  If
$\deg_{x_t}(x_tv)\leq \deg_{x_t} (x_ju)$,
 then $\deg_{x_t}(v)< \deg_{x_t} (u)$, because $j \neq t$. We again argue as before.
 There exists some $p$ such that $\deg_{x_p} v > \deg_{x_p} u$ and $x_t(v/x_p)\in I$.
 If $\deg_{x_p}(x_tv)> \deg_{x_p} (x_ju)$, then $x_i(x_tv/x_p)\in \mm I$, as required.
 Otherwise $p=j$ and $v_2=x_t(v/x_j)\in G(I)$ because $x_t v\in G(\mm I)$.
 Note that $\deg_{x_i}v_2=\deg_{x_i}v<\deg_{x_i}u$ and $\deg_{x_j}v_2=\deg_{x_j}(v)-1=\deg_{x_j}(u)$.
 Therefore, there exists some $l\neq j$ with $\deg_{x_l}( v_2) > \deg_{x_l}( u)$ such that $x_i(v_2/x_l) \in I$.
 Moreover, $\deg_{x_l} (x_tv)=\deg_{x_l}(v_2)> \deg_{x_l}(u)=\deg_{x_l} (x_j u)$.
 This shows that $x_i(v_2/x_l) =x_i x_t(v/x_jx_l) \in I$, and consequently $x_i x_t(v/x_l) \in \mm I$.  \\

{\bf Case 2:} Let $i=k$. Then  $\deg_{x_i} (x_ku)= \deg_{x_i} (x_iu)= \deg_{x_i} u+1$. This leads to further two subcases:\\
{\bf Case 2.1:} Let $\deg_{x_i}(v)=  \deg_{x_i} (u)$. If
$\deg_{x_t}(v)<  \deg_{x_t} (u)$,
 then there exists some $j$ such that $\deg_{x_j}(v)>  \deg_{x_j} (u)$ and $x_t(v/x_j)\in I$.
 Since $j \neq i$ and $j \neq t$, we have $\deg_{x_j}(x_tv)=\deg_{x_j}(v)> \deg_{x_j} (u)=\deg_{x_j}(x_iu)$
 and $x_ix_t(v/x_j)\in \mm I$ as required. Otherwise, $\deg_{x_t}(v)\geq  \deg_{x_t} (u)$,
 which gives $\deg_{x_t}(x_t v)>  \deg_{x_t} (x_iu)$, and then $x_i(x_tv/x_t) \in \mm I$.

{\bf Case 2.2:} Let $\deg_{x_i}(v)<  \deg_{x_i} (u)$. Then there
exists some $j$ such that $\deg_{x_j}(v)>  \deg_{x_j} (u)$ and
$x_i(v/x_j)\in I$. Since $j \neq i$,
 we have $\deg_{x_j} (x_iu)= \deg_{x_j} (u)$, therefore, $\deg_{x_j}(x_tv)\ge \deg_{x_j}(v)>\deg_{x_j}(u)= \deg_{x_j} (x_iu)$,
 then $x_i(x_tv/x_j) \in \mm I$, as required.
\end{proof}


\section{Componentwise linearity and exchange property of ideals of  $k$-covers of hypergraphs}\label{sec2}
A finite {\it hypergraph} $\H$ on the vertex set $V({\mathcal{H}})$ is a collection of edges  $E({\H})=\{ J_1, \ldots, J_m\}$ with $J_i \subseteq V({\H})$ and $J_i \neq \emptyset$ for all $i=1, \ldots,m$. Throughout the following text, we let $V(\mathcal{H})=[n]$. A weighted hypergraph is a hypergraph $\mathcal{H}$ together with an integer valued weight function $\omega: E(\mathcal{H})\rightarrow \NN$. For any $k \in \NN$, a {\em $k$-cover} of a weighted hypergraph $(\mathcal{H}, \omega)$ is a vector $c=(c_1 , \ldots, c_n) \in \NN^n$ that satisfies the condition $\sum_{i \in J} c_i \geq k \omega(J)$ for all $J \in E(\H)$. For every edge $J \in E(\H)$, let $\mm_J$ be the ideal generated by the variables $x_i$ with $i \in J$. Following \cite{HHT}, we set $I(\H, \omega):= \cap_{J \in E(\H)}\mm_J^{\omega(J)}$. The ideal $I_k(\mathcal{H}, \omega)= \cap_{J \in E(\mathcal{H})}\mm_J^{k\omega(J)}$ is called {\em ideal of $k$-covers of $(\mathcal{H},\omega)$}, as described in \cite{MM}. The ideal $I_k(\mathcal{H}, \omega)$ can also be viewed as {\em intersection of Veronese ideals}, as defined in \cite{FT}. 
Mohammadi and Moradi in \cite{MM} and Francisco and Tyul in \cite{FT} studied the componentwise linearity of $I_k(\mathcal{H}, \omega)$ for different classes of hypergraphs. The aim of this section is to extend their work and also to find under what conditions on the edges of $\mathcal{H}$, the ideal of $k$-covers of $(\H,\omega)$ satisfies some exchange property. In the following work, we discuss componentwise linearity of some classes of totally balanced weighted hypergraphs. To this end, we recall some definitions and notations. 
 
 Let $\H$ be a hypergraph. A sequence $v_1, J_1, v_2, J_2, \ldots, v_s, J_s, v_{s+1}=v_1$ of distinct edges $J_1, \ldots, J_s$ and distinct vertices $v_1, \ldots, v_s$ of $\mathcal{H}$ is called a {\em cycle} of length $s$ in $\mathcal{H}$ if $v_i,v_{i+1} \in J_i$, for all $i=1, \ldots, s$. Such a cycle is called {\em special} if no edge contains more than two vertices of the cycle. The concept of special cycles in hypergraph generalizes the notion of cycles in graphs. Hypergraphs without any special cycles of length greater than three are called {\em totally balanced hypergraphs}, see \cite[Chapter 5]{B}. In the language of simplicial comlexes, totally balanced hypergraphs correspond to simplcial forests, introduced by Faridi in \cite{Far1}.  The edge ideals of totally balanced hypergraphs possess many nice properties as noted in \cite{Far1} and \cite{Far2}. For example, it is shown in \cite[Corollaries  5.5 and  5.6]{Far2} that if $\H$ is a totally balanced hypergraph (equivalently, a simplicial forest), then the ideal of $1$-covers of $\H$ is  componentwise linear. In the following theorem, we consider a special class of totally balanced weighted hypergraphs, whose ideals of $k$-covers admit the non-pure dual exchange property, and hence they are componentwise linear.

\begin{Theorem}\label{thm:intersection}
Let $J_1,\ldots, J_s, K\subseteq [n]$  be such that $J_i\cap
J_j=\cap_{t=1}^s J_t$ for all $1\leq i<j\leq s$ and $K\cap
J_t=\emptyset$ for all $t=1, \ldots, s$. Then for any positive
integers $a_1, \ldots, a_{s},b$, the ideal $I=\mm^{a_1}_{J_1}\cap
\mm^{a_2}_{J_2}\cap\cdots \cap  \mm^{a_s}_{J_s}\cap\mm^b_K\subset
S=K[x_1, \ldots, x_n]$ has the non-pure dual exchange property.
\end{Theorem}

\begin{proof}
Let $B=\cap_{t=1}^s J_t$, and set $A_t=J_t\setminus B$ for all $t=1,
\ldots, s$. Without loss of generality, we may assume that $a_1\leq
\cdots \leq a_s$, and $J_i\neq J_j$ for all distinct $i$ and $j$.
For any monomial $w\in S$, we write $w=w_1\cdots w_s w'w''h$ such that
$\supp(w_t) \subseteq A_t$ for all $t=1, \ldots, s$,   $\supp(w')
\subseteq B$ and $\supp(w'') \subseteq K$, and $h$ is some monomial
with $\supp(h) \cap [ J_1\cup \cdots \cup J_s  \cup K ]= \emptyset$.
Then $w \in I$ if and only if for each $t= 1, \ldots,  s$,
\begin{equation}\label{eq2}
\deg(w')+\deg(w_t)\geq a_t  \quad \text{and}\quad \deg(w'') \geq b.
\end{equation}
Moreover, if $w \in G(I)$, then $h=1$ and using $a_1\leq \cdots \leq
a_s$, we have the following:

\begin{enumerate}
\item[(i)] $\deg(w'') = b$ and $\deg(w')\leq a_s$. Moreover, $\deg(w') =a_s$ if and only if $\deg(w_t)=0$, for all $t=1, \ldots, s$.
\item[(ii)] For each $t=1, \ldots, s$, $\deg(w_t) \neq 0 $ if and only if $\deg(w') < a_t$.
Moreover, if $\deg(w') \leq  a_t$ for some $t=1 , \ldots, s$, then
$\deg(w') +\deg(w_t)  =a_t$.

Indeed, if $\deg(w') < a_t$, then  $\deg(w_t) \neq 0 $ is a direct
consequence of (\ref{eq2}).  On the other hand, if $\deg(w_t) \neq 0
$, and $\deg(w') \geq  a_t$, then $w/w_t \in I$ because it satisfies
the inequalities in (\ref{eq2}) and strictly divides $w$, a
contradiction to the assumption that $w \in G(I)$.

\item[(iii)] Combining (ii) with  $\deg(w')\leq a_s$, we obtain $\deg(w') +\deg(w_s)  =a_s$. This gives
\begin{equation}\label{eq:degree}
\deg(w)=\sum_{t=1}^s\deg(w_t)+\deg(w')+\deg(w'')=
\sum_{t=1}^{s-1}\deg(w_t)+a_s+b.
\end{equation}
Moreover, we have $a_s+b \leq \deg(w) \leq
\sum_{i=1}^sa_i+b$.
\end{enumerate}

Let $u,v \in G(I)$ and $u=u_1\cdots u_s u'u''$ and $v=v_1\cdots v_s
v'v''$ as described above. Now, we show that $I$ has the non-pure dual
exchange property. To do this, let $\deg(u)\leq \deg(v)$ and
$\deg_{x_i}(v) < \deg_{x_i}(u)$ for some $i$. We need to show that
there exists $j$ such that $\deg_{x_j}(v) > \deg_{x_j}(u)$ and
$x_i(v/x_j)\in I$. We consider the following cases:

{\bf Case 1:} Let $i\in K$. Since $ \deg(u'')=\deg(v'')=b$, it
follows that there exists some $j \in K$ such that  $\deg_{x_j}(v) >
\deg_{x_j}(u)$. Then for any such $j$, the monomial $x_i(v/x_j)$
satisfies both inequalities in (\ref{eq2}) and hence $x_i(v/x_j) \in
I$.

{\bf Case 2:} Let $i\in B$.  Since $u, v\in G(I)$ and $
\deg(u'')=\deg(v'')=b$, there exists some $j \in \cup_{t=1}^s J_t$
such that  $\deg_{x_j}(v) > \deg_{x_j}(u)$. Then for any such $j$,
the monomial $x_i(v/x_j)$ satisfies both inequalities in (\ref{eq2})
and hence $x_i(v/x_j) \in I$.

{\bf Case 3:} Let $i\in A_p$, for some $p=1, \ldots, s$.  Then $0<
\deg(u_p)$ which together with (i) gives that $\deg(u')<a_s$.

{\bf Subcase 3.1:} Let $a_s+b=\deg(v)$. Since $\deg(u) \leq
\deg(v)$, it follows from (iii) that $a_s+b=\deg(u)$. Then due to
(\ref{eq:degree}), we have
$\sum_{t=1}^{s-1}\deg(u_t)=0=\sum_{t=1}^{s-1}\deg(v_t)$, that is
$\deg(u_t)=0=\deg(v_t)$, for all $t=1, \ldots, s-1$. Since $0<
        \deg(u_p)$, we conclude $p=s$, that is $i \in A_s$.  We can write
$u=u_su'u''$ and $v=v_sv'v''$.

If there exists some  $j \in A_s$ such that  $\deg_{x_j}(v) >
\deg_{x_j}(u)$, then  $x_i(v/x_j)$ satisfies both inequalities in
(\ref{eq2}) and hence $x_i(v/x_j) \in I$, as required. If no such
$j$ exists, then $\deg(v_s)<\deg(u_s)$, which in return gives
$\deg(v')>\deg(u')$. Then there exists some $j \in B$ such that
$\deg_{x_j}(v) > \deg_{x_j}(u)$. The first inequality of (\ref{eq2})
together with $\deg(u_t)=0$, for all $t=1, \ldots, s-1$ gives
$\deg(u')=\deg(u')+\deg(u_t)\geq  a_t$. This gives
$\deg(v')>\deg(u')\geq a_t$ for all $t=1, \ldots, s-1$. Then for
$v$, the inequalities in (\ref{eq2}) takes the following form: for
all $t=1, \ldots, s-1$, we have
\begin{equation}\label{eq:degv} \deg(v')+\deg(v_t)=\deg(v')> a_t,\;
\deg(v')+\deg(v_s)= a_s\quad \text{and}\quad \deg(v'') \geq b.
\end{equation}
Due to (\ref{eq:degv}), the monomial $x_i(v/x_j)$ satisfies both
inequalities in (\ref{eq2}) and hence $x_i(v/x_j) \in I$, as
required.

{\bf Subcase 3.2:} Now, we consider the final case when
$a_s+b<\deg(v)$.  Then following (i) and (\ref{eq:degree}) gives
$\deg(v')<a_s$. On the other hand, we have $\deg(u')<a_s$. Now, let
$k$ and $\ell$ be the minimum integers such that $\deg(u')< a_k$ and
$\deg(v')< a_\ell$. If $\ell=s$, then $a_{t}\leq \deg(v')$ for all
$t=1, \ldots, s-1$. Then using (ii) and (\ref{eq:degree}) gives
$\deg(v)=a_s+b$, a contradiction. Therefore, $\ell<s$.

{\bf Claim:}  $\deg(v')\leq \deg(u')$.

Assume that the claim holds. Since $0<\deg(u_p)$, using (ii), we
have $\deg(u')<a_p$ and $a_p-\deg(u')= \deg(u_p)$. Now, using the
assumption that the claim holds, we have $\deg(v')\leq\deg(u')<a_p$.
Again from (ii), we obtain that
 $\deg(u_p)=a_p-\deg(u')\leq a_p-\deg(v')=\deg(v_p)$.
Therefore, there exists some $j\in A_p$, such that $\deg_{x_j}(v) >
\deg_{x_j}(u)$, and for any such $j$, the monomial $x_i(v/x_j)$
satisfies both inequalities in (\ref{eq2}). Hence $x_i(v/x_j) \in
I$, as required. Now, it only remains to prove the claim.

Proof of claim: If $\ell<k$, then by the minimality of $\ell$ and
$k$, we have $\deg(v')<a_\ell \leq \deg(u')< a_k$. This gives,
$\deg(v') < \deg(u')$, as required.

Now, suppose that $k \leq \ell$.  Using (ii) and (\ref{eq:degree}),
we have
\begin{equation*}
\deg(v)=\sum_{t=\ell}^{s-1}\deg(v_t)+a_s+b    \text{ and } \deg(u)=
\sum_{t=k}^{s-1}\deg(u_t)+a_s+b
\end{equation*}
which gives
\begin{equation}\label{eq:min}
0\leq \deg(v)-\deg(u)= \sum_{t=\ell}^{s-1} \deg(v_t) -
\sum_{t=k}^{s-1}\deg(u_t).
\end{equation}
If $k=\ell$, then using (\ref{eq:min}) together with (ii) and
$a_\ell\leq\cdots \leq a_{s-1}$ gives
\begin{equation*}
\begin{split}
0 \leq  \deg(v)-\deg(u)&=\sum_{t=\ell}^{s-1} \deg(v_t) - \sum_{t=\ell}^{s-1}\deg(u_t) \\
&= \sum_{t=\ell}^{s-1} [a_t-\deg(v')] - \sum_{t=\ell}^{s-1} [a_t-\deg(u')]\\
&=  (s-\ell)[\deg(u')-\deg(v')].
\end{split}
\end{equation*}
Since $s-\ell>0$, the claim holds. On the other hand, if $k<\ell$,
then $\deg(u')<a_k\leq \deg(v')< a_\ell$ gives $\deg(u') <
\deg(v')$. Moreover, using (\ref{eq:min}) together with (ii) and
$a_\ell\leq\cdots \leq a_{s-1}$ gives
\begin{equation*}
\begin{split}
0\leq   \deg(v)-\deg(u)&=\sum_{t=\ell}^{s-1}\deg(v_t)- \sum_{t=k}^{\ell-1}\deg(u_t)- \sum_{t=\ell}^{s-1}\deg(u_t)\\
&= \sum_{t=\ell}^{s-1}[ \deg(v_t) - \deg(u_t)]- \sum_{t=k}^{\ell-1}\deg( u_t)\\
&=  (s-\ell)[\deg(u')-\deg(v')]- \sum_{t=k}^{\ell-1}\deg( u_t).
\end{split}
\end{equation*}
Then $(s-\ell)[\deg(u')-\deg(v')]\geq  \sum_{t=k}^{\ell-1}\deg(
u_t)$. But $s-\ell>0$ and $\deg(u')-\deg(v') <0$, we obtain $
\sum_{t=k}^{\ell-1}\deg( u_t)<0 $, which is false. This shows that
$k<\ell$ does not hold. This completes proof of claim.
\end{proof}

In \cite{MM}, it is shown that the ideals discussed in above theorem are in fact weakly polymatroidal. By setting $K=\emptyset$ in Theorem~\ref{thm:intersection}, we recover \cite[Corollary 3.2]{FT}. 

\begin{Corollary}
Let $J_1,J_2\subseteq
[n]$ and let $a$ and $b$ be positive integers.
Then $I=\mm^a_{J_1}\cap \mm^b_{J_2} \subset K[x_1, \ldots, x_n]$
satisfies the non-pure dual exchange property.
\end{Corollary}

It is shown in \cite[Theorem 4.3]{FT} that the ideal of $k$-covers of any weighted hypergraph with three edges is componentwise linear. However, such ideals need not to be componentwise polymatroidal. This prompts the natural question of whether these ideals satisfy some exchange property on their generators. In the following example, we observe that these ideals may not possess the non-pure dual exchange property, even when the hypergraph is totally balanced. 
\begin{Example}\label{exp:3ideals}{\em 
Let $\mm_{J_1}=(x_1 ,
	x_2)$, $\mm_{J_2}=(x_2 , x_3,x_4)$, $\mm_{J_3}=(x_4, x_5)$, $  I=
	\mm_{J_1}\cap \mm_{J_2}\cap \mm_{J_3}$. It is easy to see that $J_1, J_2, J_3$ determine a totally balanced hypergraph. Consider the monomial
	$u=x_2x_5$ and $v=x_1x_4$ in $G(I)$. We have $\deg(u)=\deg(v)$ and
	$\deg_{x_5}(v)<\deg_{x_5}(u)$, but $x_1x_5$ and $x_4x_5$ do not
	belong to $I$. This shows that $I$ does not have the non-pure dual exchange property.
}
\end{Example}

Next, we show that for any totally balanced weighted hypergraph $(\H, \omega)$ with three edges, the ideal of $k$-covers of $(\H, \omega)$ is weakly polymatroidal with some conditions on $\omega$. The exchange property satisfied by polymatroidal ideals was generalized by Hibi and Kokubo in \cite{KH} through the introduction of so-called weakly polymatroidal ideals. These ideals, as studied and defined in \cite{KH}, are generated in the same degree. This definition was further extended by Mohammadi and Moradi in \cite{MM}, where they considered weakly polymatroidal ideals that are not necessarily generated in the same degree. Below we recall the definition of weakly polymatroidal ideals as given in \cite{MM}.
	\begin{Definition}\label{weakly}
		Let $u, v \in S=K[x_1, \ldots, x_n]$ be two monomials with
		$u=x_1^{a_1}\cdots x_n^{a_n}$ and $v=x_1^{b_1}\cdots x_n^{b_n}$. Let
		$>_{\lex}$ denote the {\em pure lexicographical order} induced on $S$ with
		respect to the total order of variables $x_{i_1}>\cdots > x_{i_n}$.
		Then $u>_{\lex} v$ if and only if $\deg_{x_{i_k}} (u)=
		\deg_{x_{i_k}} (v)$ for some $k=1, \ldots, t-1$ and $\deg_{x_{i_t}}
		(u)> \deg_{x_{i_t}} (v)$.
		
		A monomial ideal $I\subset S$ is called {\em weakly polymatroidal}
		if for every two monomials $u=x_1^{a_1}\cdots x_n^{a_n}$ and
		$v=x_1^{b_1}\cdots x_n^{b_n}$ with $u>_{\lex} v$ and $\deg_{x_{i_k}}
		(u)= \deg_{x_{i_k}} (v)$ for some $k=1, \ldots, t-1$ and
		$\deg_{x_{i_t}} (u)> \deg_{x_{i_t}} (v)$, there exists some $j$ such
		that $x_{i_t}(v/x_{i_j}) \in I$.
	\end{Definition}

	\begin{Remark}{\em
			Weakly polymatroidal ideals and ideals with the non-pure dual
			exchange property are not necessarily equivalent.  For example, consider any weakly polymatroidal ideal
			$I$ generated in degree $d$ such that $I$ is not polymatoridal. According to
			\cite[Proposition 1.4]{BQ}, such an ideal
			$I$ does not satisfy the non-pure dual exchange property .
			On the other hand, let $I=(x_2x_3, x_1^2x_3)$. This ideal satisfies the non-pure dual
			exchange property. However $I$ is not weakly polymatroidal.
		}
	\end{Remark}
	It is important to note that the definition of weakly polymatroidal ideal depends on the total order of variables.

\begin{Theorem}\label{thm:3weakly}
	Let $J_1,J_2, J_3\subseteq [n]$ such that $J_1 \cap J_3 = \emptyset$
	and $S=K[x_1, \ldots, x_n]$. Then for any positive integers $a_1\geq
	a_2\geq a_3$, the ideal $I= \mm^{a_1}_{J_1}\cap  \mm^{a_2}_{J_2}
	\cap  \mm^{a_3}_{J_3}$ is weakly polymatroidal.
\end{Theorem}
\begin{proof}
	If either $J_1 \cap J_2 = \emptyset$ or $J_3 \cap J_2 = \emptyset$,
	then the assertion holds due to Theorem~\cite[Theorem 2.4]{MM}.
	Now, assume that $J_1 \cap J_2\neq \emptyset$ and $J_3 \cap J_2\neq
	\emptyset$. To show $I$ is weakly polymatroidal, first we introduce
	a total order on the variables of $S$. We set the following
	notations.
	\begin{enumerate}
		\item $J'_{1}:= J_{1}\cap J_{2}$, $J'_{2}:= J_{2}\cap J_{3},$
		\item  $J''_1 := J_1\setminus J_{2}$, $J''_2:= J_2\setminus (J_1\cup J_3)$, and  $J''_3 := J_3\setminus J_{2}$.
	\end{enumerate}
	
	With above notations, we have $J_1=J''_1 \sqcup J'_1$,
	$J_2=J'_1\sqcup  J'_2  \sqcup J''_2 $, and $J_3=J'_{2} \sqcup
	J''_3$. Set $K=[n]\setminus (J_1 \cup J_2 \cup J_3)$. For any variable $x_i$, the index $i$ belongs to a unique
	set among $J'_1, J'_2, J''_1, J''_2,J''_3, K$.  For each of these sets,
	choose any ordering of the variables with indices in it, and then we
	set a total order on all variables such that $x_i<x_j$ if and only
	if  in the following list, the set to which $i$ belongs to appear
	before the set to which $j$ belongs to
	\[
	J'_1, J''_1, J'_2, J''_3,  J''_2, K.
	\]
	In other words, we first list the variables with indices in $J'_1$,
	then we list the variables with indices in $J''_1$. After that we
	continue with listing the variables with indices in $J'_2$, and so
	on. Once all the variables with indices in $J_1\cup J_2 \cup J_3$
	are listed, then continue listing the rest of the variables in $K$.
	
	Given any monomial $w$, we write $w=w'_1w'_{2}w''_1w''_2w''_3h$ such
	that $\supp(w'_t) \subseteq J'_t$ for $t=1,2$ and $\supp(w''_t)
	\subseteq J''_t$ for $t=1,2,3$, and $\supp(h) \subseteq K$. Then $w
	\in I $ if and only if
	\begin{equation}\label{eq:inclusion}
		\begin{split}
			a_1&\leq \deg(w'_1)+\deg(w''_1),\\
			a_2&\leq  \deg(w'_{1})+\deg(w'_2)+\deg(w''_{2}),\\
			a_3&\leq  \deg(w'_2)+\deg(w''_{3}).
		\end{split}
	\end{equation}
	Moreover, if $w \in G(I)$, then $ a_1= \deg(w'_1)+\deg(w''_1)$.
	To justify this, we argue in similar way as in Theorem~\ref{thm:4ideals}.
	Suppose that $a_1<\deg(w'_1)+\deg(w''_1)$. If  $\deg(w'_1)> a_1$ then by setting $\tilde{w}'_1$ to
	be any monomial that divides $w'_1$ with $a_1=\deg(\tilde{w}'_1)$, we obtain $\tilde{w}=\tilde{w}'_1w'_{2}w''_1w''_2w''_3 \in I$
	because it satisfies (\ref{eq:inclusion}). Also, $\tilde{w}$ strictly divides $w$, a contradiction to $w \in G(I)$.
	Therefore, $\deg(w'_1)\leq a_1$ holds. Then $a_1<\deg(w'_1)+\deg(w''_1)$ gives $0\leq a_1-\deg(w'_1)< \deg(w''_1) $.
	Now, we set $\tilde{w}''_1$ to be any monomial that divides $w''_1$ with $\deg(\tilde{w}''_1)=a_1-\deg(w'_1)$,
	and take $\tilde{w}=w'_1w'_2\tilde{w}''_1w''_2w''_3$. Then $\tilde{w}$ satisfies (\ref{eq:inclusion})
	and strictly divides $w$, a contradiction to $w \in G(I)$. Hence, $w \in G(I)$ gives  $a_1=\deg(w'_1)+\deg(w''_1)$.
	
	From above discussion, we conclude that $w \in G(I)$ if and only if
	$h=1$ and
	\begin{equation}\label{eq:gen}
		\begin{split}
			a_1&= \deg(w'_1)+\deg(w''_1),\\
			a_2&\leq  \deg(w'_{1})+\deg(w'_2)+\deg(w''_{2}),\\
			a_3& \leq \deg(w'_2)+\deg(w''_{3}).\\
		\end{split}
	\end{equation}
	
	Let $u=u'_1 u'_2u''_1u''_2u''_3$, and $v=v'_1v'_2v''_1v''_2v''_3$ be
	two elements in $G(I)$ with $u >_{\lex} v$. Let $i \in [n]$ be such that
	$\deg_ {x_i}(u)>\deg_{x_i}(v)$ and for all $x_r>x_i$, we have $\deg_
	{x_r}(u)=\deg_{x_r}(v)$. To show $I$ is weakly polymatroidal, we
	need to find a suitable $j$ such that  $x_i>x_j$ and $x_i(v/x_j) \in
	I$. We consider the following cases:
	
	Case 1: Let  $i \in \supp(u'_1)\subseteq J'_1$.  If there exists
	some $j \in \supp(v'_1)$ for which $x_i>x_j$, then $x_i(v/x_j) \in
	I$ because $x_i(v/x_j)$ satisfies the inequalities in
	(\ref{eq:inclusion}). Otherwise, we have $\deg(v'_1)<\deg(u'_1)$.
	Since $\deg(u'_1)+\deg(u''_1)=a_1= \deg(v'_1)+\deg(v''_1)$ due to
	(\ref{eq:gen}), we obtain $\deg(v''_1)>\deg(u''_1)\geq 0$, and there
	exists some $j \in \supp(v''_1)$. Then $x_i>x_j$ and $x_i(v/x_j) \in
	I$ because $x_i(v/x_j)$ satisfies the inequalities in
	(\ref{eq:inclusion}).
	
	Case 2: Let  $i \in \supp(u''_1)\subseteq J''_1$. Then $u'_1=v'_1$
	and it follows from the equality in (\ref{eq:gen}) that
	$\deg(u''_1)=\deg(v''_1)$. Since $\deg_{x_i}(u)>\deg_{x_i}(v)$, we
	obtain some $j  \in \supp(v''_1)$ for which
	$\deg_{x_j}(u)<\deg_{x_j}(v)$. This gives $x_j<x_i$. Furthermore,
	$x_i(v/x_j) \in I$ because  $x_i(v/x_j)$ satisfies the inequalities
	in (\ref{eq:inclusion}).
	
	Case 3: Let  $i \in \supp(u'_2)\subseteq J'_2$. Then $u'_1=v'_1$ and $u''_1=v''_1$.
	We claim that either there exists some $j  \in \supp(v'_2)$ with
	$x_j<x_i$ or $\supp(v''_2)\cup \supp(v''_3)\neq \emptyset$. If the
	claim holds then there exists some j such that $x_j<x_i$ and
	$x_i(v/x_j) \in I$ because  $x_i(v/x_j)$ satisfies the inequalities
	in (\ref{eq:inclusion}), and we are done.
	
	Suppose that no such  $j  \in \supp(v'_2)$  exists and
	$\supp(v''_2)\cup\supp(v''_3)=\emptyset$.
	Then $v'_2$ strictly divides $u'_2 $ and $\deg(v''_2)=0=\deg(v''_3)$. This gives $v=u'_1u''_1v'_2$ and $v$ strictly divides $u$. This implies that $u$
	is not a minimal generator of $I$, a contradiction. Hence, the claim
	holds.
	
	Case 4: Let  $i \in \supp(u''_3)$. Then $u'_1=v'_1$, $u''_1=v''_1$ and $u'_2=v'_2$. Assume that $\deg(u'_2) \geq a_3$.
	Then the monomial $u/  u''_3 \in I$ because
	it satisfies the inequalities in (\ref{eq:inclusion}). Since $\deg(u''_3)\neq 0$,
	it yields $u / u''_3$ strictly divides $u$,  a contradiction to $u$ being a
	minimal generator of $I$. Therefore, we have $\deg(u'_2)<a_3$.
	This also shows in the case when $\deg(u''_3)\neq 0$, the last inequality of (\ref{eq:gen}) is indeed an equality $\deg(u'_2)+\deg(u''_3)=a_3$,
	otherwise we would have $u \not\in G(I)$, a contradiction.
	
	Since $v'_2=u'_2$, it yields $\deg(v'_2)=\deg(u'_2)<a_3$, and consequently $\deg(v''_3)=\deg(u''_3)=a_3-\deg(u'_2)$.
	Now, since $u''_3\neq v''_3$ and $u>_{\lex}
	v$, it follows that there exists some $j  \in \supp(v''_3)$ with  $x_j<x_i$.
	Furthermore, $x_i(v/x_j) \in I$ because $x_i(v/x_j)$ satisfies the inequalities in (\ref{eq:inclusion}), and we are done.
	
	Case 5: Let $i \in \supp(u''_2)$. Then $u'_1=v'_1$, $u''_1=v''_1$, $u'_2=v'_2$, $u''_3=v''_3$. Since $v$ does
	not divide $u$ because $u \in G(I)$, we obtain some $j  \in
	\supp(v''_2)$ for which $\deg_{x_j}(u)<\deg_{x_j}(v)$. Then
	$x_j<x_i$ and $x_i(v/x_j) \in I$ because  $x_i(v/x_j)$ satisfies the
	inequalities in (\ref{eq:inclusion}).
\end{proof}

\begin{Remark}{\em The conditions $a_1\geq a_2 \geq a_3$ and $J_1\cap
		J_3=\emptyset$	introduced in Theorem~\ref{thm:3weakly} are needed. We justify this in the following examples.
	\begin{enumerate}
		\item 	Let $a_1, a_2$ and $a_3$ be any positive integers. Then $I$
	described in Theorem~\ref{thm:3weakly} need not to be weakly
	polymatroidal with respect to the total order described in
	Theorem~\ref{thm:3weakly}. To see this, take $\mm_{J_1}=(x_1 ,
	x_2)$, $\mm_{J_2}=(x_2 , x_3,x_4)$, $\mm_{J_3}=(x_4, x_5)$, and $ I=
	\mm_{J_1}^2\cap \mm_{J_2}^3\cap \mm_{J_3}^2$. Then
	$x_2>x_1>x_4>x_5>x_3$, and $x_2^3x_5^2>x_2^2x_4^2$. But
	$x_2^3x_4\notin I$.
	\item The condition $J_1\cap
	J_3$ is necessary in Theorem~\ref{thm:3weakly}. In other words, if $\mathcal{H}$ is not a totally balanced hypergraph then its ideal of $k$-covers need not to be weakly polymatroidal. For example, take  $\mm_{J_1}=(x_1 , x_2,x_3)$,
	$\mm_{J_2}=(x_3 , x_4,x_5)$, $\mm_{J_3}=(x_1, x_5,x_6)$, and $    I=
	\mm_{J_1}^2\cap \mm_{J_2}^2\cap \mm_{J_3}^2$. We claim that $I$ is
	not weakly polymatroidal with respect to any total order on the
	variables $x_1, \ldots, x_6$. To prove this, first observe that for
	$I$ to be weakly polymatroidal we must have $x_5>x_4$. To see this
	consider the monomials $x_2^2x_5^2$ and $x_2^2x_4^2x_6^2$. If
	$x_4>x_5$, then $x_2^2x_4^2x_6^2>x_2^2x_5^2$, but $x_4x_2^2x_5,
	x_4x_2x_5^2, x_6x_2^2x_5, x_6x_2x_5^2 \not\in I$. Similarly, one can
	see $x_3>x_4$, by comparing $x_3^2x_6^2$ and $ x_2^2x_4^2x_6^2$.
	Indeed, if $x_4>x_3$, then $x_2^2x_4^2x_6^2>x_3^2x_6^2 $, but
	$x_2x_3x_6^2, x_2x_3^2x_6, x_4x_3x_6^2, x_4x_3^2x_6 \not\in I$.
	
	Since the role of $x_1$, $x_3$ and $x_5$ is symmetrical, we may
	choose to set $x_1>x_3>x_5$. Then together with $x_5>x_4$ and
	$x_3>x_4$, we have $x_1^2x_4^2>x_1x_3x_5$ but $x_1^2x_3$ and
	$x_1^2x_5$ are not in $I$. Hence $I$ is not weakly polymatoridal.
		\end{enumerate}
	}
\end{Remark}

In \cite[Remark 4.4]{FT}, authors provided an example of a hypergraph with four edges (which can also be viewed as the ideal of a tetrahedral curve),  that is not componentwise linear. We recall this example. Let $I= (x_1, x_2) \cap (x_2,x_3) \cap (x_3, x_4) \cap (x_4,x_1)$. Then $I$ is an ideal of $1$-covers, or simply the vertex cover ideals, of a cycle of length 4. It can be easily verified that $I$ is not componentwise linear. Below, we investigate the case of totally balanced hypergraphs with four edges under certain conditions. In fact, we consider a class of hypergraphs on four edges which takes form of a path graph on five vertices if all the edges have size two. In the sequel, we denote the union of two disjoint sets $A$ and $B$ with $A\sqcup B$.

\begin{Theorem}\label{thm:4ideals}
    Let $J_1,J_2,J_3, J_4 \subset [n]$ such that
    \begin{enumerate}
        \item $J_1 \cap J_3=J_1\cap J_4 = J_2\cap
        J_4=\emptyset$,
        \item         	$J_2 \subseteq J_1 \cup J_3$ and $J_3 \subseteq J_2 \cup J_4$.
    \end{enumerate}
    Then for any positive integer $a$, the ideal $I= \mm^{a}_{J_1}\cap
    \mm^{a}_{J_2} \cap  \mm^{a}_{J_3}\cap \mm^{a}_{J_4}\subset
    S=K[x_1,\ldots,x_n]$ satisfies the non-pure dual exchange property.
\end{Theorem}
\begin{proof}
    First, we introduce some notations as follows:
    \begin{enumerate}
        \item[(i)] for each $t=1, 2,3$, set $J'_{t}:= J_{t}\cap J_{t+1}$,
        \item[(ii)]  $J''_1 := J_1\setminus J_{2}$, and  $J''_4 := J_4\setminus J_{3}$.
    \end{enumerate}

    With above notations, we have $J_1= J'_1\sqcup J''_1$,  $J_2=J'_{1}
    \sqcup J'_2$,  $J_3=J'_{2} \sqcup J'_3$, and $J_4=J'_{3} \sqcup
    J''_4$. For any variable $x_i$, the index $i$ belongs to a unique
    set among $J'_1, J'_2, J'_3, J''_1, J''_4$. Given any monomial $w \in S$,
    we write $w=w'_1w'_2w'_{3}w''_1 w''_{4}h$ such that $\supp(w'_t)
    \subseteq J'_t$ for $t=1, 2 ,3$ and $\supp(w''_t) \subseteq J''_t$
    for $t=1,  4$, and $h$ is any monomial with support in $[n]\setminus
    (J_1\cup \cdots \cup J_4)$. Then $w \in I $ if and only if
    \begin{equation}\label{eq:inI1}
        \begin{split}
            a&\leq \deg(w'_1)+\deg(w''_1),\\
            a&\leq  \deg(w'_{1})+\deg(w'_2),\\
            a&\leq  \deg(w'_{2})+\deg(w'_3),\\
            a&\leq \deg(w'_{3})+\deg(w''_4).
        \end{split}
    \end{equation}
    Moreover, $w \in G(I)$ if and only if $h=1$ and
    \begin{equation}\label{eq:inI2}
        \begin{split}
            a&=\deg(w'_1)+\deg(w''_1)=\deg(w'_{3})+\deg(w''_4),\\
            a&\leq  \deg(w'_{1})+\deg(w'_2),\\
            a&\leq  \deg(w'_{2})+\deg(w'_3).
        \end{split}
    \end{equation}

To see why above statement holds, suppose that
$a<\deg(w'_1)+\deg(w''_1)$. We first observe that  $\deg(w'_1)\leq
a$. Indeed, if  $\deg(w'_1)> a$, then by setting $\tilde{w}'_1$ to
be any monomial that divides $w'_1$ with $a=\deg(\tilde{w}'_1)$, we
obtain $\tilde{w}=\tilde{w}'_1w'_2w'_{3}w''_1 w''_{4} \in I$ because
it satisfies (\ref{eq:inI1}). Also, $\tilde{w}$ strictly divides
$w$, a contradiction to $w \in G(I)$. Therefore, $\deg(w'_1)\leq a$
holds. Then $a<\deg(w'_1)+\deg(w''_1)$ gives $0\leq a-\deg(w'_1)<
\deg(w''_1) $. Set $\tilde{w}''_1$ to be any monomial that divides
$w''_1$ with $a-\deg(w'_1)=\deg(\tilde{w}''_1)$, and take
$\tilde{w}=w'_1w'_2w'_{3}\tilde{w}''_1w''_{4}$. Then $\tilde{w}$
satisfies (\ref{eq:inI1}) and strictly divides $w$, a contradiction to
$w \in G(I)$. Therefore, we have  $a=\deg(w'_1)+\deg(w''_1)$ as
claimed in (\ref{eq:inI2}). Arguing in a similar way we can show
that $a=\deg(w'_{3})+\deg(w''_4)$. Moreover, the inequalities in
(\ref{eq:inI2}) are due to (\ref{eq:inI1}).

Using first equality in (\ref{eq:inI2}), we obtain
\begin{equation}\label{deg1}
        \deg(w)=\deg(w'_1)+\deg(w'_2)+\deg(w'_3)+\deg(w''_1)+\deg(w''_4)=2a+\deg(w'_2).
\end{equation}
Let $u=u'_1 u'_2u'_3u''_1u''_4$, and $v=v'_1v'_2v'_3v''_1v''_4$ such
that $\deg(u)\leq \deg(v)$. Let $i$ be such that
$\deg_{x_i}(u)>\deg_{x_i}(v)$. To show $I$ has the non-pure dual
exchange property, we need to find a suitable $j$ such that
$\deg_{x_j}(u)<\deg_{x_j}(v)$ and $x_i(v/x_j) \in I$. We consider
the following cases:

Case 1: Let  $i \in J'_t$ where $t\in \{1,3\}$ . First let $t=1$.
The discussion is same when $t=3$. If there exists some $j  \in
J'_1$ for which $\deg_{x_j}(u)<\deg_{x_j}(v)$, then $x_i(v/x_j) \in
I$ because  $x_i(v/x_j)$ satisfies the inequalities in
(\ref{eq:inI1}). Otherwise, we have $\deg(v'_1)<\deg(u'_1)$. Since
$\deg(u'_1)+\deg(u''_1)=a= \deg(v'_1)+\deg(v''_1)$, we obtain
$\deg(v''_1)>\deg(u''_1)$, and there exists some $j  \in J''_1$ for
which $\deg_{x_j}(u)<\deg_{x_j}(v)$.  Then $x_i(v/x_j) \in I$
because  $x_i(v/x_j)$ satisfies the inequalities in (\ref{eq:inI1}).

Case 2: Let $i \in J'_2$. Since $\deg(u)\leq \deg(v)$, we obtain
$\deg(u'_2)\leq \deg(v'_2)$ because of (\ref{deg1}). This gives some
$j\in J'_2$ for which $\deg_{x_j}(u)<\deg_{x_j}(v)$. Then
$x_i(v/x_j) \in I$ because  $x_i(v/x_j)$ satisfies the inequalities
in (\ref{eq:inI1}).

Case 3: Let  $i \in J''_t$ where $t\in \{1,4\}$ . First let $t=1$.
The discussion is same when $t=4$. If there exists some $j  \in
J''_1$ for which $\deg_{x_j}(u)<\deg_{x_j}(v)$, then $x_i(v/x_j) \in
I$ because  $x_i(v/x_j)$ satisfies the inequalities in
(\ref{eq:inI1}). Otherwise, we have $\deg(v''_1)<\deg(u''_1)$. Since
$\deg(u'_1)+\deg(u''_1)=a= \deg(v'_1)+\deg(v''_1)$, we obtain
$\deg(u'_1)<\deg(v'_1)$, and there exists some $j  \in J'_1$ for
which $\deg_{x_j}(u)<\deg_{x_j}(v)$.  Let  $w=x_i(v/x_j)$. To show
$w\in I$, it is enough to show that it satisfies  the second
inequality in (\ref{eq:inI1}) because $w$ satisfies the other
inequalities in (\ref{eq:inI1}). Since $\deg(u)\leq \deg(v)$, we
obtain $\deg(u'_2)\leq \deg(v'_2)$ because of (\ref{deg1}). Using
$\deg(u'_1)< \deg(v'_1)$ gives
    \[
    a\leq  \deg(u'_{1})+\deg(u'_2) < \deg(v'_{1})+\deg(v'_2).
    \]
Therefore, $    a< \deg(v'_{1})+\deg(v'_2)$. Since
$\deg(w'_{1})=\deg(v'_{1})-1$ and $\deg(w'_{2})=\deg(v'_{2})$, we
obtain $   a\leq  \deg(w'_1)+\deg(w'_2)$, as required.
\end{proof}
One can modify the proof of above theorem to obtain the following result that we state as corollary of above theorem. 
\begin{Corollary}
	Let $J_1, J_2, J_3 \subset [n]$ such that $J_1 \cap J_3 =  \emptyset$, and $J_2 \subseteq J_1 \cup J_3$, then for any positive integer $a$, the ideal $I=\mm_{J_1}^a \cap  \mm_{J_2}^a \cap \mm_{J_3}^a$ satisfies the non-pure dual exchange property.
\end{Corollary}

Note that in Example~\ref{exp:3ideals} the ideal $I$ failed to have the non-pure dual exchange property because $J_2 \not\subseteq J_1\cup J_3$. In the following remark, we justify the conditions on $J_i$'s  given in Theorem~\ref{thm:4ideals}.
\begin{Remark}{\em
		\begin{enumerate}
			\item The ideals described in Theorem~\ref{thm:4ideals} need not to have
    the non-pure dual exchange property if any of the equality in the statement (2) of  Theorem~\ref{thm:4ideals} is
    violated. For example $\mm_{J_1}=(x_1 , x_2)$,
    $\mm_{J_2}=(x_2,x_3)$, $\mm_{J_3}=(x_3 , x_4,x_5)$, $\mm_{J_4}=(x_5,
    x_6)$, and $I=\mm_{J_1}^2\cap \mm_{J_2}^2\cap \mm_{J_3}^2\cap
    \mm_{J_4}^2$. Let $ u=x_1x_2x_3x_5^2$ and $v=x_2^2x_4^2x_6^2$. Then
    $u,v \in G(I)$, $\deg(u)<\deg(v)$ and $\deg_{x_1}(v)<\deg_{x_1}(u)$
    but $x_1x_2^2x_4^2x_6^2 / x_i$ for any $i=2,4,6$ is not in $I$.
    
\item 
    The ideals described in Theorem~\ref{thm:4ideals} need not to have the
    non-pure dual exchange property if we consider $I=
    \mm^{a_1}_{J_1}\cap  \mm^{a_2}_{J_2} \cap  \mm^{a_3}_{J_3}\cap
    \mm^{a_4}_{J_4} $ with $a_i\neq a_j$ for some $i$ and $j$. For
    example $\mm_{J_1}=(x_1 , x_2)$, $\mm_{J_2}=(x_2,x_3)$,
    $\mm_{J_3}=(x_3 , x_4)$, $\mm_{J_4}=(x_4, x_5)$, and $I=
    \mm^{2}_{J_1}\cap  \mm^{2}_{J_2} \cap  \mm^{2}_{J_3}\cap \mm_{J_4}
    $. Then $u=x_1x_2x_3x_4$, $v=x_2^2x_4^2$ are in $G(I)$ and
    $\deg_{x_1}(v)<\deg_{x_1}(u)$, but $x_1x_2^2x_4, x_1x_2x_4^2 \not\in
    I$. Therefore, $I$ does not have the non-pure dual exchange property.

 \item    The ideals described in Theorem~\ref{thm:4ideals} need not to have
    the non-pure dual exchange property if we extend the construction to hypergraphs with five edges. For example $\mm_{J_1}=(x_1 , x_2)$,
    $\mm_{J_2}=(x_2,x_3)$, $\mm_{J_3}=(x_3 , x_4)$, $\mm_{J_4}=(x_4,
    x_5)$, $\mm_{J_5}=(x_5, x_6)$ and $I= \mm_{J_1}\cap  \mm_{J_2} \cap
    \mm_{J_3}\cap \mm_{J_4} \cap \mm_{J_5}$. Then $u=x_1x_3x_5$,
    $v=x_2x_4x_5$ are in $G(I)$ and $\deg_{x_1}(v)<\deg_{x_1}(u)$, but
    $x_1x_2x_5, x_1x_4x_5 \not\in I$. Therefore, $I$ does not have
    the non-pure dual exchange property.
        	\end{enumerate}
}
\end{Remark}

\end{document}